\documentclass[12pt]{article}
\usepackage{amsmath}
\usepackage{amssymb}
\usepackage{stmaryrd}
\usepackage{color}
\usepackage{ wasysym }
\usepackage{color}
\newcommand{\supp}{\mathop{\mathrm{\supp}}}

   \newtheorem{theorem}{Theorem}[section]
   \newtheorem{lemma}[theorem]{Lemma}
   \newtheorem{proposition}[theorem]{Proposition}
   \newtheorem{corollary}[theorem]{Corollary}

\newcommand{\ben}{\begin{enumerate}}
\newcommand{\een}{\end{enumerate}}

\newcommand{\bt}{\begin{theorem}}
\newcommand{\et}{\end{theorem}}
\newcommand{\bl}{\begin{lemma}}
\newcommand{\el}{\end{lemma}}
\newcommand{\bc}{\begin{corollary}}
\newcommand{\ec}{\end{corollary}}
\newcommand{\bp}{\begin{proposition}}
\newcommand{\ep}{\end{proposition}}
\newcommand{\br}{\begin{remark}}
\newcommand{\er}{\end{remark}}

\newcommand{\bpf}{\begin{proof}}
\newcommand{\epf}{\end{proof}}

\newcommand{\be}{\begin{equation}} 
\newcommand{\ee}{\end{equation}}
\newcommand{\beq}{\begin{eqnarray}}
\newcommand{\eeq}{\end{eqnarray}}
\newcommand{\ba}{\begin{array}}
\newcommand{\ea}{\end{array}}
\newcommand{\bi}{\begin{itemize}}
\newcommand{\ei}{\end{itemize}}

\newcommand{\comm}[1]{}

\newcommand \qed {\hskip 6pt\vrule height6pt width5pt depth1pt \bigskip}
\newfont{\msbm}{msbm10 scaled\magstep1}%blackboardbold
\newfont{\msbms}{msbm7 scaled\magstep1} %blackboardbold   subscript
   \newenvironment{proof}[1][Proof]{\begin{trivlist}
   \item[\hskip \labelsep {\bfseries #1}]}{\end{trivlist}}
   \newenvironment{definition}[1][Definition]{\begin{trivlist}
   \item[\hskip \labelsep {\bfseries #1}]}{\end{trivlist}}
   
   \newenvironment{remark}[1][Remark]{\begin{trivlist}
   \item[\hskip \labelsep {\bfseries #1}]}{\end{trivlist}}

   \comm{\newcommand{\qed}{\nobreak \ifvmode \relax \else
      \ifdim\lastskip<1.5em \hskip-\lastskip
      \hskip1.5em plus0em minus0.5em \fi \nobreak
      \vrule height0.75em width0.5em depth0.25em\fi}}

 \numberwithin{equation}{section}

\usepackage{fancyhdr,amsfonts, amsmath}

\newcommand{\rarrow}{\rightarrow}

\begin{document}

\vspace{ 1 in}

\title{Weak solutions to the Navier - Stokes equations}
\author{Ira Herbst}
\date\today

\maketitle
\emph{Mathematics Subject Classification (2020): 35Q35, 35D30}.\\

\emph{Key words: weak solution, H\"{o}lder continuity, Kato approximation.} \\

ABSTRACT: We prove the equivalence of being a Leray-Hopf weak solution of the Navier-Stokes equations in $\mathbb{R}^m, \  m\ge 3$, to satisfying a well known integral equation and use this equation to derive some regularity properties of these weak solutions.

\section{Introduction}
The purpose of this work is to prove that the property of being a Leray-Hopf \cite {JL, EH}  weak solution, $u(t)$, of the Navier-Stokes equations in $\mathbb{R}^m, m \ge 3$, is equivalent to $u(t)$ satisfying the well-known perturbation equation (see Theorem \ref{equivalence})

\begin{equation}\label{perturbation}
u(t) = e^{t\Delta}u_0 -\mathbb{P}  \int_0^t e^{(t-s)\Delta} (u(s)\cdot \nabla u(s)) ds ;  \ a.e.\ t \ge 0
\end{equation} 

including $t = 0$ and to prove some properties of these solutions (see Section \ref{conseq}).  The equivalence is known with some definition (probably weaker than that of Leray and Hopf \cite{JL,EH})  of weak solution (see \cite{LR} and for an earlier proof in $\mathbb{R}^3$ \cite{FLRT}).  We use the definition of weak solution for which Leray \cite {JL}  and Hopf \cite {EH} proved the existence of global in time solutions given an arbitrary divergence free initial velocity $u_0 \in L^2$.  
In  (\ref{perturbation}) $\mathbb{P}$ is the Leray projector onto divergence free fields and $\Delta$ is the Laplacian.  We call (\ref{perturbation}) the perturbation equation because in some sense it treats the Navier-Stokes equations as a perturbation of the heat equation (which in many ways it is not).  We make the following definition:

 \begin{definition}
$u(t)$ is a weak solution of the Navier-Stokes equations in $\mathbb{R}^m$ if for each $T > 0$, $u \in L^\infty([0,T];L^2(\mathbb{R}^m)) \cap L^2([0,T];H^1(\mathbb{R}^m)), \nabla_x \cdot u(t,x) = 0$ (distributional derivatives) and 
\begin{equation} \label{weakeqn} 
\begin{split}
\int_0^t [- (u(s),\partial_s \phi(s))  + (\nabla u(s), \nabla \phi(s)) + (u(s)\cdot \nabla u(s), \phi(s)) ]ds  =  \\
  (u_0,\phi(0)) - (u(t), \phi(t)); \ a.e. \ t \ge 0, \text{including} \ t = 0.
  \end{split}
\end{equation}

We do not incorporate the energy inequality (see \cite{JL, EH}) in our definition of weak solution. \\

Here $(\cdot,\cdot)$ is the usual inner product in $L^2$ where we are abbreviating $(u,v) = \sum_{i=1}^m (u_i,v_i) $ and $(\nabla u, \nabla v) = \sum _{i,j} (\partial_{x_i}u_j,\partial_{x_i}v_j)$ when $u$ and $v$ are vector fields.  Sometimes we use the notation $(f,g)$ when possibly neither $f$ nor $g$ is in $L^2$ but the product $fg$ is in $L^1$.  The test functions, $\phi$, are in $C_0^\infty([0,T) \times \mathbb{R}^m)$ and satisfy $\nabla_x \cdot \phi(s,x) = 0$.  The function $u_0 \in L^2$ is the initial fluid velocity.  Setting $t=0$ in (\ref{weakeqn}) we find $u(0) = u_0$.  The space $H^1$ is the $L^2$ - Sobolev space of functions $f$ with $||\nabla f||_2^2 + ||f||_2^2 < \infty$.  Weak solutions (with the definition above) are known to exist globally, i.e. for all $t\ge 0$,  given an arbitrary divergence free initial velocity $u_0 \in L^2(\mathbb{R}^m)$, see \cite {JL, EH}.
Weak solutions are apriori only defined for almost every $t$, but as is shown in \cite{RRS} for $m=3$ we can choose $u(t)$ for each $t\in [0,\infty) $ so that $u$ is weakly continuous in $L^2$.  For the reader's convenience we give a proof for $m\ge 3$ in Section \ref{testf}. 
\end{definition}
\
In Section \ref{testf} we  extend the set of test functions to $\{\phi \in \mathcal{S}(R^{m+1}): \text{supp} \phi \subset K \times \mathbb{R}^m, \nabla_x \cdot \phi(s,x) = 0 \}$, where $K$ is a compact subset of $(-\infty, T)$ and $\mathcal{S}(\mathbb{R}^{m+1})$ is the Schwartz space of $C^\infty(\mathbb{R}^{m+1})$ functions of rapid decrease.
Then after another small extension, in Section \ref{pert}  we derive (\ref{perturbation}) and show that the validity of this equation is equivalent to the Leray-Hopf definition of weak solution given above.  
In Section 4 we use the perturbation equation to show that $v(t) = u(t) - e^{t\Delta}u_0$ is in $L^1(\mathbb{R}^m)$ and has nearly one distributional space derivative in $L^1$ (in a sense to be made precise).  In addition $v(t)$ is $L^1$ - H\"older continuous in time of order $\alpha$ for any $\alpha < 1/2$ and $L^1$ - H\"older continuous in space of order $r$ for any $r < 1$.    It would be very interesting to know if this regularity can be improved. The implications of these results for $L^p, 1<p <2$ are given.

\section{Extending the space of test functions} \label{testf} 

We use Kato's \cite{K} method of approximating a divergence free field, $f$, by one of compact support.  Let $z_R(x) = \tilde z(R^{-1} \log(|x|^2 +1))$ where $\tilde z \in C^\infty (\mathbb{R})$ with $ \tilde z(t) = 1$ for  $ t  \le 1$ and $\tilde z(t)  = 0$ for $t \ge 2$.  For $k=0$ and $1$, define 
$$Q_k(f)(x) = \int_0^1 t^{m-1-k} f(tx) dt,$$  Kato's approximation of $f$ is
$$f_R(x)_j  = z_R(x) f_j(x)  + \sum_{i=1}^m\partial_i z_R(x) (x\wedge Q_1(f)(x))_{ij} .$$
Here $(x\wedge g(x))_{ij} = x_i g_j(x) - x_j g_i(x)$.  A calculation which we omit gives  

$$\nabla \cdot f_R = z_R \nabla \cdot f + (x\cdot \nabla z_R) \ Q_0(\nabla \cdot f).$$  Thus if $f$ is divergence free, $f_R$ has compact support and is divergence free.  We estimate $||f_R - f||_p $ and $ ||\nabla(f_R - f)||_p$.  First we have ($p'$ is the dual index to $p$) 

$$\partial_iQ_1(f) = Q_0(\partial_i f) ,$$
$$||Q_k(f)||_p \le (-k+ m/p')^{-1} ||f||_p ; \ \infty \ge p > (1-k/m)^{-1}.$$
$$|\partial_i z_R(x)| \le cR^{-1} (|x| + 1)^{-1} , |\partial_i\partial_j z_R(x)| \le c'R^{-1} (|x| + 1)^{-2} .$$

We thus have for $\infty \ge p > m'$,
$$||f_R -f||_p \le ||(1-z_R)f||_p + c\sup_{x  \in \mathbb{R}^m}|x| |\nabla z_R(x)|\ ||Q_1(f)||_p \le ||(1-z_R)f||_p+ c'R^{-1}||f||_p ,$$
$$ ||\nabla(f_R - f)||_p \le ||(1-z_R)\nabla f ||_p + c R^{-1} (||f||_p + ||\nabla f||_p)$$

We will use these bounds for $f(\cdot) = \phi(s,\cdot)$ with $\phi$ taken from the set $  \mathcal{T}_1: = \{\phi \in \mathcal{S}(R^{m+1}): \text{supp} \phi \subset K \times \mathbb{R}^m, \nabla_x \cdot \phi(s,x) = 0 \}$, where $K$ is a compact subset of $(-\infty, T)$. As a final extension, we would like to be able to use test functions of the form $\phi(s,\cdot)  \in \mathcal{T}: = \{\mathbb{P} \phi_0(s,\cdot): \phi_0 \in  \mathcal{S}(R^{m+1}), \text{supp} \phi_0 \subset K \times \mathbb{R}^m \}$ with $K$ a compact subset of $(-\infty, T)$.  That is we drop the requirement that the test function be in $\mathcal{S}(R^{m+1})$ with zero divergence for all $s$ and replace these test function with the stated functions which may be less smooth in Fourier space.
 We use the cutoff function $\chi_r(\xi) = e^{-r/|\xi|^2}$, $r>0$, defined to be $0$ when $\xi = 0$.  If the vector function $f \in \mathcal{S}(\mathbb{R}^m) $ then the function $f_r$ with Fourier transform given by $\chi_r (\xi) \hat f(\xi)$ has the property that $\mathbb{P} f_r $ is also in $ \mathcal{S}(\mathbb{R}^m) $ since $\widehat f_r$ and all its derivatives are zero at the origin so that no singularity arises from applying $\mathbb{P} $.  With $H^k (\mathbb{R}^m) $ the $L^2$ Sobolev space with norm $||f||^2_{H^k}  = \sum _{|\alpha| \le k} \int |\xi^{\alpha} \hat f(\xi)|^2  d\xi = 
 \sum_{|\alpha | \le k}\int  |\partial^{\alpha}f(x)|^2 dx $ we find
 \begin{equation} \label{Papprox} 
 \lim_{r \to 0 } || \mathbb{P} (f - f_r) ||_{H^k(\mathbb{R}^m) } =0.
 \end{equation} 

 Before replacing the test functions in (\ref{weakeqn})  with a larger class we follow the ideas in \cite{RRS} to show that by making a particular definition of $u(t)$  for all time we can ensure that $u$ is $L^2$ weakly continuous.  First choosing $\alpha(s) \in C_0^\infty(\mathbb{R})$ with $\alpha(s) = 1$ for $s\in [t_1,t_2]$ and letting $\tilde{\phi}(s,x)= \alpha(s) \phi(x)$ with  $\phi \in C_0^\infty(\mathbb{R}^m))$ with $\nabla \cdot \phi =0$ we subtract (\ref{weakeqn}) with $t=t_1$ from (\ref{weakeqn}) with $t = t_2$ to obtain 
  
 \begin{equation} \label{t1t2eqn} 
 \int_{t_1}^{t_2} [(\nabla u(s),\nabla \phi) + (u(s)\cdot \nabla u(s),\phi)]ds = (u(t_1), \phi) - (u(t_2),\phi)
 \end{equation} 

 for a.e. $t_1$ and $t_2$ but including $t_1 = 0$.  With $m'$ the dual index to $m$
 $$||u(s)\cdot\nabla u(s)||_{m'} \le c||\nabla u(s)||_2^2.$$ 
 
 This follows from 
 $$||u\cdot \nabla u||_{m'} \le c'||u||_{(2^{-1} - m^{-1})^{-1}} ||\nabla u||_2 \le c  ||\nabla u||_2^2,$$
 where the first inequality is just H\"{o}lder's inequality and the last is a Sobolev inequality.
 
 It is convenient to replace $\int _{t_1}^{t_2} (u(s) \cdot \nabla u(s),\phi)ds$ with  $\int _{t_1}^{t_2} (\mathbb{P} (u(s)\cdot\nabla u(s)),\phi)ds$ and then use the estimates above to replace the class of test functions in (\ref{t1t2eqn}) by $\phi \in \mathcal{S}(\mathbb{R}^m)$.  That this is possible follows easily.   We now would like to show that we can replace the class $\mathcal{S}(\mathbb{R}^m)$ with $B = H^1(\mathbb{R}^m) \cap L^m(\mathbb{R}^m)$, in other words that $ \mathcal{S}(\mathbb{R}^m)$ is dense in $B$.  Let $\chi , \psi \in C_0^\infty(\mathbb{R}^m)$ with $\chi(x) \in [0,1]$, $\chi(x) = 1$ for $|x| \le 1$, and $\psi \ge 0$ with $\int \psi dx = 1$.  Define $\chi_\delta(x) = \chi(\delta x),
 \psi_\epsilon(x) = \epsilon^{-m} \psi(x/\epsilon)$.  Then if $\phi \in B$ , $\phi_{\epsilon,\delta} = \chi_\delta (\psi_\epsilon* \phi ) \in \mathcal{S} $ and converges to $\phi$ in both $H^1$ and $L^m$.

 Then we have 
 \begin{equation} \label{geqn}
 u(t) - u_0 = \int_0^t g(s) ds, a.e. 
 \end{equation} 
 
 where $g(s)  \in B^* $  for almost every $s$.  We have 
 $$\langle g(s), \phi \rangle =  -(\nabla u(s), \nabla \phi) - (\mathbb{P}(u(s)\cdot \nabla u(s)), \phi)$$
 
  so that $|\langle g(s), \phi \rangle| \le c ||\nabla u(s)||_2^2 ||\phi||_m + ||\nabla u(s)||_2 ||\phi||_{H^1}  $ . Thus $||g(s)||_{B^*}$ is integrable over $[0,T]$ for any $T$.  We choose $u(t)$ so that (\ref{geqn})  is satisfied for all $t$.  Then $u(t) - u_0 \in C([0,T], B^*) , T > 0$.  We would like to show that for our choice of $u(t)$, $||u(t)||_2 \le ||u||_{L^\infty([0,T];L^2)} = :c_0$ for all $t \in [0,T]$.  Choose $t_n \rarrow t $ with $||u(t_n)||_2 \le c_0$.  Then there is some subsequence of $\{t_n\}$ for which $u(t_n)$ converges weakly to some $v \in L^2$.  We relabel the subsequence so that $u(t_n)$ converges weakly to $v$ in $L^2.$  Suppose $f \in \mathcal{S}(\mathbb{R}^m)  (\subset B) $.  Then $( u(t_n),f) \rightarrow (v,f) $ and since $u(t) -u_0 \in C([0,T];B^*)$, $(u(t_n),f) \rightarrow (u(t),f)$. Thus $u(t) = v$ and since $ c_0 \ge \limsup ||u(t_n)||_2 \ge ||v||_2$ we have $||u(t)||_2 \le c_0$.  Using the convergence of $u(t_n)-u_0$ to $u(t)-u_0$ in $B^* $ for any sequence $t_n \rightarrow t$ we have $(u(t_n) , f) \rightarrow (u(t),f)$ for $f \in \mathcal{S}$.  Since $||u(t_n) - u(t)||_2 \le 2 c_0$ this follows for all $f \in L^2$.  Thus we have for our choice of $u(t)$,
 
 \begin{lemma} \label{weakcont} 
 $ u(\cdot) $ is $L^2$ - weakly continuous.  
  \end{lemma} 
 
 We now show how to replace the test functions in (\ref{weakeqn}) with a larger class.  We first see how to replace those test functions in $C_0^\infty([0,T) \times \mathbb{R}^m)$ and satisfying $\nabla_x \cdot \phi(s,x) = 0$ with functions $\phi \in \mathcal{T}_1$.  Consider each term in (\ref{weakeqn})  where $0 <t  <T$:
 With $\phi \in \mathcal{T}_1$, as $R \rightarrow \infty$, 
 $$|\int_0^t (u(s), \partial_s(\phi_R(s,\cdot)) - \phi(s,\cdot)) ds|  \le \int_0^t ||u(s)||_2 (||(1-z_R)\partial_s \phi(s)||_2 + c'R^{-1} ||\partial_s\phi(s)||_2)ds  \rightarrow 0.$$
 Similarly 
 $$|  \int_0^t (\nabla u(s),\nabla (\phi_R(s,\cdot) - \phi(s,\cdot))) ds| $$  $$\le \int_0^t ||\nabla u(s)||_2 (||(1 - z_R) \nabla \phi(s)||_2+ c'' R^{-1}(||\phi(s)||_2 + ||\nabla \phi(s)||_2)) \rightarrow 0 $$
where we use the Schwarz inequality and the fact that $\int _0^t ||\nabla u(s)||_2^2 ds < \infty$.  For the $u\cdot \nabla u$ term we have 

$$ |\int_0^t (u(s)\cdot \nabla u(s), \nabla (\phi_R(s,\cdot) - \phi(s,\cdot))  ds | \le \int_0^t ||u(s)\cdot\nabla u(s)||_{m'} ||\nabla(\phi_R(s,\cdot) - \phi(s,\cdot)) ||_m ds.$$
As we showed above 
$$ ||\nabla(\phi_R(s,\cdot) - \phi(s,\cdot)) ||_p  \rightarrow 0 $$ for all $p \in [2,\infty]$ and this is uniform for $s\in[0,t]$ so this term $\rightarrow 0$. 
The remaining two terms in (\ref{weakeqn}) also converge to what they are supposed to and thus we can use $\phi \in \mathcal{T}_1$ in (\ref{weakeqn}).
Basically the same ideas allows us to replace test functions in $\mathcal{T}_1$ with test functions in $\mathcal{T}$.  We need to show that the terms in (\ref{weakeqn}) with $\phi(s)$ replaced with $\mathbb{P}(\phi_r(s,\cdot) - \phi(s,\cdot) )$ tend to $0$. Here we use that $\mathbb{P}$ is bounded on $L^p(\mathbb{R}^m)$ for $1< p < \infty$.  We omit the details.  Using the $L^2$ - weak continuity of $u(t)$ it is easy to see that (\ref{weakeqn}) is true for all $t$.  

\section{The perturbation equation} \label{pert}

In this section we prove the following theorem.
\begin{theorem}\label{equivalence}
Suppose $u$ is a weak solution (see (\ref{weakeqn}) of the Navier-Stokes equations so that $u \in L^\infty([0,T];L^2(\mathbb{R}^m)) \cap L^2([0,T];H^1(\mathbb{R}^m)), \nabla_x \cdot u(t,x) = 0$ for all $T > 0$ and with $u(0) = u_0, \  0\le t \le T$, 
 \begin{equation}
 \begin{split}
  \int_0^t [- (u(s),\partial_s \phi(s))  + &(\nabla u(s), \nabla \phi(s)) + (u(s)\cdot \nabla u(s), \phi(s)) ]ds    =   \\
                & (u_0,\phi(0)) - (u(t), \phi(t)), \text{a.e.} \ t>0
  \end{split}
 \end{equation}

for all $\phi \in C_0^\infty([0,T) \times \mathbb{R}^m)$ satisfying $\nabla_x \cdot \phi(s,x) = 0$.  Then $u$ satisfies the perturbation equation 
\begin{equation}\label{perturbation'}
u(t) = e^{t\Delta}u_0 -\mathbb{P}  \int_0^t e^{(t-s)\Delta} (u(s)\cdot \nabla u(s)) ds 
\end{equation} 
for a.e. $t\ge 0$, including $t=0$.
Conversely if for all $T>0$, $u \in L^\infty([0,T];L^2(\mathbb{R}^m)) \cap L^2([0,T];H^1(\mathbb{R}^m)), \nabla_x \cdot u(t,x) = 0 $ 
and the perturbation equation (\ref{perturbation'}) holds for a.e. $t \ge 0$ including $t=0$, 
then $u$ is a weak solution of the Navier-Stokes equations.
\end{theorem}

\begin{proof}

 Fix $t$ with $0 < t < T , \epsilon \in (0,t),$  and $\alpha_\epsilon \in C_0^\infty((-\infty,T))$ with  $\alpha_\epsilon(s) = 1$ for $s \in [0,t-\epsilon]$  and $\alpha_\epsilon(s) = 0 $ for $s\ge t$.  In addition we take $0  \ge \alpha'_{\epsilon} (s) $ for $s \in (t-\epsilon, t)$.  We define a test function $\phi_\epsilon (s, \cdot) = \mathbb{P}e^{(t-s)\Delta}\alpha_\epsilon( s)  f(\cdot)$ where $f \in \mathcal{S}(\mathbb{R}^m)$.  Notice that  $\phi_\epsilon \in \mathcal{T}$.  We consider 
$$\int_0^t (u(s), \partial_s \phi_\epsilon(s)) ds = \int_{t-\epsilon}^t(u(s),e^{(t-s)\Delta}f)\alpha_\epsilon'(s) ds - \int_0^t (u(s), e^{(t-s)\Delta} \Delta f) \alpha_\epsilon(s) ds.$$ 
From  (\ref{weakeqn}) it follows that
$$ -\int_{t-\epsilon}^t (u(s), e^{(t-s)\Delta} f) \alpha'_\epsilon(s) ds+\int_0^t (u(s)\cdot\nabla u(s), \mathbb{P} e^{(t-s)\Delta}f) \alpha_\epsilon(s) ds = (u_0, e^{t\Delta}f).$$
The first term does not change if we choose $u(s)$ for all $s$ such that $u(s)$ is weakly continuous (see Lemma \ref{weakcont}).  Then since $u(s)$ converges weakly to $u(t)$ as $s \rightarrow t$  and $e^{(t-s) \Delta} f$ converges strongly to $f$ as $s\uparrow t$, $(u(s), e^{(t-s)\Delta} f)$ converges to $(u(t),f)$ as $s\uparrow t$.  Since $\int_{t-\epsilon}^t |\alpha'_\epsilon(s)| ds  = - \int_{t-\epsilon}^t \alpha'_\epsilon(s) ds = 1$, the first term above converges to $(u(t),f)$ as $\epsilon \rightarrow 0$.  In the second term we use the Lebesgue dominated convergence theorem , $|\alpha_\epsilon (s)| \le 1$ in $[0,t]$ and the fact that the integrand is bounded by $c||u(s)\cdot \nabla u(s)||_{m'} ||\mathbb{P} e^{(t-s)\Delta} f||_m \le c' ||\nabla u(s)||_2^2 ||f||_m$ which is independent of $\epsilon$ and integrable.  Thus $\alpha_\epsilon$ can be replaced by $1$ in the limit.  

It follows that for the weakly continuous version of $u$, for every $t > 0$,
$$(u(t),f) + (\int_0^t (\mathbb{P} e^{(t-s)\Delta} u(s) \cdot \nabla u(s),f) = (e^{t\Delta} u_0,f) $$ 
for all $f \in \mathcal{S}(\mathbb{R}^m)$ and thus we have (\ref{perturbation}) for all $t \ge 0$ for the weakly continuous version and for almost every $t \ge 0$ for any version.  

Conversely suppose $u \in L^\infty([0,T];L^2(\mathbb{R}^m)) \cap L^2([0,T];H^1(\mathbb{R}^m)), \nabla_x \cdot u(t,x) = 0$ and (\ref{perturbation}) holds for almost all $t>0$ and for  $t=0$.  Choose a test function $ \phi \in C_0^\infty( (-\infty, T) \times \mathbb{R}^m) $ with $\nabla_x \cdot \phi(t,x) = 0$ and for $0<t<T$ consider 
$$(u(t),\partial_t\phi(t)) = (e^{t\Delta} u_0 - \int_0^t e^{(t-s)\Delta} (u(s)\cdot \nabla u(s)) ds, \partial_t\phi(t)) = $$ 
$$ (u_0, -e^{t\Delta} \Delta \phi(t) + \partial_t (e^{t\Delta} \phi(t)) ) + \int_0^t (u(s) \cdot \nabla u(s), e^{(t-s)\Delta} \Delta \phi(t)) ds - \int_0^t (u(s) \cdot \nabla u(s), \partial_t (e^{(t-s)\Delta} \phi(t)) ds =$$
$$ - (u(t),\Delta \phi(t)) + (u_0, \partial_t (e^{t\Delta} \phi(t)) ) - \int_0^t (u(s) \cdot \nabla u(s) ,\partial_t (e^{(t-s)\Delta} \phi(t))) ds.   $$
Integrating on the interval $[0,t]$ we have 
$$\int_0^t (u(s), \partial_s\phi(s))ds = $$ $$\int_0^t (\nabla u(s),\nabla \phi(s))ds + \int_0^t \partial_s(u_0, e^{s\Delta} \phi(s))ds- \int_0^t \int_0^s (u(\tau) \cdot \nabla u(\tau), \partial_s(e^{(s-\tau) \Delta} \phi(s)))d\tau ds = $$
$$\int_0^t (\nabla u(s),\nabla \phi(s))ds + (u_0, e^{t\Delta} \phi(t))   - \int_0^t (e^{(t-\tau)\Delta} (u(\tau)\cdot \nabla u(\tau)), \phi(t) )d\tau  -(u_0,\phi(0))$$
$$  +\int_0^t (u(\tau) \cdot \nabla u(\tau) ,\phi(\tau)) d\tau = $$

$$ \int_0^t (\nabla u(s),\nabla \phi(s))ds + ( u(t),\phi(t)) - (u_0, \phi(0)) +\int_0^t(u(s) \cdot \nabla u(s), \phi(s)) ds .$$ 
From (\ref{perturbation}), the last equality holds a.e. This is (\ref{weakeqn}).
\end{proof}
$\blacksquare$
\section{Consequences of the perturbation equation: $L^1$ regularity} \label{conseq}

In this section we show that if $u(t)$ is a weak solution to the Navier-Stokes equations, then $v(t) = u(t) -e^{t\Delta}u_0$ is in $L^1(\mathbb{R}^m)$ with some mild smoothness properties in the $x$ variable; more explicitly in some sense almost one distributional derivative in $L^1$.  In addition, in the $L^1$ norm, we prove H\"older continuity in the $t$ variable of degree $\alpha$  for any $\alpha <1/2$.  We consider the consequences of these results for $L^p, 1 < p < 2$.

We will work with the spaces $\mathcal{L}^p_r(\mathbb{R}^m) \subset L^p(\mathbb{R}^m) $ defined as $(1-\Delta)^{-r/2}L^p(\mathbb{R}^m)$ with norm  $||f||_{\mathcal{L}^p_r} = ||(I - \Delta)^{r/2}f||_p$ (see \cite {St} and \cite{Si}) for a discussion of these spaces). The next Proposition shows that $v(t) \in \mathcal{L}^1_r(\mathbb{R}^m)$ for $0 \le r < 1$.  And see an extension in Corollary \ref{Holder}.

\begin{theorem}\label{L1onederivative} 
Suppose $u(t)$ is a weak solution to the Navier-Stokes equations in dimension $m \ge 3$.  Then  $v(t) = u(t) - e^{t\Delta}u_0 \in C([0,\infty); \mathcal{L}^1_r(\mathbb{R}^m))$  for $ 0 \le r < 1$ with norms bounded uniformly for $t$ and $r$ in compact subsets of  $[0,\infty)$ and $[0,1)$ respectively.
\end{theorem}

\begin{proof}
We have $(I-\Delta)^{r/2}v(t) = -\mathbb{P} \int_0^t(I -\Delta)^{r/2} e^{(t-s)\Delta} u(s) \cdot \nabla u(s) ds$.  The fact that $\mathbb{P}$ is not a bounded operator on $L^1$ complicates the proof.  Since $\mathbb{P} = I - (I - \mathbb{P})$ we begin by estimating  $(I -\mathbb{P}) (I - \Delta)^{r/2} K_t(x)$ where $K_t$ is the integral kernel for the operator $e^{t\Delta}$ in $\mathbb{R}^m$, $K_t(x-y) = (4\pi t)^{-m/2} e^{-|x-y|^2/4t}$.  We have $\mathcal{F}((I - \mathbb{P}) f)_i(\xi)  = \sum_j  \xi_i\xi_j |\xi|^{-2} \mathcal{F} f_j(\xi)$ where $\mathcal{F}$ is the Fourier transform. Thus working in $x$ - space with $y=x/\sqrt t$ we have

$$\partial_{x_i}\partial_{x_j} (I - \Delta)^{r/2}\Delta^{-1} K_t(x) = -(2\pi)^{-m} t^{-(m+r)/2} \partial_{y_i}\partial_{y_j} \int (t + |\eta|^2)^{r/2} |\eta|^{-2} e^{-|\eta|^2} e^{iy\cdot \eta} d\eta.$$

A bit more calculation gives
\be \label{1-p} 
 (I-\mathbb{P})_{i,j}  (I - \Delta)^{r/2}K_t(x) = ct^{-(m+r)/2} (\omega_i\omega_j g'_{r,t}(|y|) + |y|^{-1}(\delta_{i,j} - \omega_i\omega_j)g_{r,t}(|y|)),
 \ee
 where $\omega_i = y_i/|y|$ and

 $$g_{r,t}(w) =  \int_0^\infty \int_{-1}^1 (t+ s^2)^{r/2} s^{m-2} e^{-s^2} e^{isw\lambda} \lambda (1-\lambda^2)^{(m-3)/2} d\lambda ds.$$
 
 Here $c$ is an $m$ dependent constant.
 
 It is shown in the Appendix that the $C^\infty$ function $g_{r,t}$ and its derivatives obey the following estimates 
 
 \be \label{grtestimates} 
 |g^{(n)}_{r,t}(w)|  \le C_n(1+|w|)^{-(m+n-1)},
 \ee
for all $r \ge 0$, uniformly for $t$ in compact sets of $[0,\infty)$.

We need to estimate the following $L^1$ norm:
$$\int|\int_0^t\int \partial_{x_i}\partial_{x_j} (I-\Delta)^{r/2}\Delta^{-1} K_s(x-z) f_j(t-s,z) dz ds|dx$$

where $f_j(s,z) = u(s,z)\cdot \nabla u_j(s,z)$.  For a reason that will be apparent only at the end of the proof we decouple the $j$ indices and replace $f_j$ by $f_a$ where now $i,j,a$ can be all different.

We first look at the term involving the derivative $g'_{r,t}$ in (\ref{1-p}):

$$\int | \int_0^t s^{-(m+r)/2}\int  \omega_i \omega_j g'_{r,s}(|y|)f_a(t-s,z)dz ds|dx$$
where $y = (x-z)/\sqrt{s}$ and $\omega_i = y_i/|y|$.  We write

$$f_a(s,z) = \sum_k (\partial/\partial z_k) (u_k(s,z) u_a(s,z))$$

and integrate by parts in the $z$ integral.  We use $(\partial /\partial z_k) |y| = O(s^{-1/2}), (\partial/\partial z_k) \omega_i = O(|y|^{-1} s^{-1/2})$ to get 

$$ |(\partial/\partial z_k)(\omega_i \omega_j g'_{r,t}(|y|)| \le c s^{-1/2} |y|^{-1}(1 + |y|)^{-m} $$
which gives
\be \label{lastL1bound} 
\begin{split} 
\int_0^t \int s^{-(m+r)/2}(\int |(\partial /\partial z_k)  [\omega_i \omega_j g'_{r,s}(|y|)]|dx) |u_k(t-s,z) u_a(t-s,z)|   dz ds\le \\
\int_0^t s^{-(1+r)/2} ||u(t-s)||_2^2 ds \le c(1-r)^{-1} ||u||_{L^\infty([0,T], L^2))}^2.
\end{split} 
\ee
For the term involving $|y|^{-1}g_{r,s}$  in (\ref{1-p}) we again integrate by parts and use 

$$  | (\partial /\partial z_k)( |y|^{-1} g_{r,s}(|y|)| \le c s^{-1/2} (|y|^{-2} |g_{r,s}(|y|)| + |y|^{-1}|g'_{r,s}(|y|)|) \le C s^{-1/2} |y|^{-2}(1+|y|)^{-m+1} .$$

This term is then estimated in the same way as the previous term involving $g'_{r,t}$ which results in 

$$||(I - \mathbb{P})  \int_0^t(I -\Delta)^{r/2} e^{(t-s)\Delta} u(s) \cdot \nabla u(s) ds||_{L^1} \le C_{r,t} $$

where $C_{r,t}$ is bounded uniformly for $r$ and $t$ in compacts of $[0,1)$ and $[0,\infty)$ respectively.  

To show that 

$$||\mathbb{P}  \int_0^t(I -\Delta)^{r/2} e^{(t-s)\Delta} u(s) \cdot \nabla u(s) ds||_{L^1} $$

is bounded we note $\mathbb{P} = I - (I-\mathbb{P})$ so we just need to deal with the identity term.  But $\sum_j \partial_{x_j} \partial_{x_j} \Delta^{-1} = I$ so that the bound proved for 
$$(I-\mathbb{P})_{i,j} \int_0^t(I -\Delta)^{r/2} e^{(t-s)\Delta} u(s) \cdot \nabla u_a(s) ds $$ immediately gives the desired estimate.  

It remains to prove the continuity in $t$.  With $\delta > 0$ we have 

$$|| \mathbb{P}  \int_0^{t+\delta}(I -\Delta)^{r/2} e^{(t+ \delta-s)\Delta} u(s) \cdot \nabla u(s) ds - \mathbb{P}  \int_0^t(I -\Delta)^{r/2} e^{(t-s)\Delta} u(s) \cdot \nabla u(s) ds||_1 \le $$

$$  || \mathbb{P}  \int_0^t (I -\Delta)^{r/2} (e^{(t+ \delta-s)\Delta} - e^{(t-s)\Delta} ) u(s) \cdot \nabla u(s) ds||_1 + ||\mathbb{P}\int_t^{t+\delta}  (I -\Delta)^{r/2}e^{(t+ \delta-s)\Delta} u(s) \cdot \nabla u(s) ds||_1.  $$

The first term is just 
$$||(e^{\delta\Delta} - I)(I - \Delta)^{r/2} v(t)||_1$$ 
which tends to zero as $\delta \downarrow 0$ since $\{e^{t\Delta} \}$ is a $C_0 $ semi-group on $L^1$.  For the second term we refer to (\ref{lastL1bound}).  The estimate we need is the last term where $t$ is replaced by $t+\delta$ and $0$ is replaced by $t$.  We obtain 
$$ ||\mathbb{P}\int_t^{t+\delta}  (I -\Delta)^{r/2}e^{(t+ \delta-s)\Delta} u(s) \cdot \nabla u(s) ds||_1 \le  c\int_t^{t+\delta}(t+\delta -s)^{-(1 +r)/2} ds  ||u||_{L^\infty([0,T], L^2))}^2 = $$ 
$$ = 2c(1-r)^{-1} \delta^{(1-r)/2} ||u||_{L^\infty([0,T], L^2))}^2   $$
where $T$ can be taken any number larger than say $t+1$ (where $\delta <1$).  

When we consider $v(t-\delta) - v(t)$ with $\delta > 0$ the proof needs the following lemma (which we will also make use of in the corollary which follows):
\bl
Suppose $\epsilon \in (0,1]$.  Then 
$$ ||(I - e^{h\Delta})(I-\Delta)^{-\epsilon} ||_{L^1 \rightarrow L^1} \le C h^\epsilon$$
\el
\bpf
We write 
$$(I - e^{h\Delta})(I - \Delta)^{-\epsilon} = \Gamma(\epsilon)^{-1} (I - e^{h\Delta}) \int_0^\infty t^{\epsilon - 1} e^{-t} e^{t\Delta} dt = $$

$$\Gamma(\epsilon)^{-1}\big( \int_0^\infty t^{\epsilon - 1} e^{-t} e^{t\Delta} dt -  \int_0^\infty t^{\epsilon - 1} e^{-t} e^{(t+h)\Delta} dt\big)=$$

$$\Gamma(\epsilon)^{-1}\big( \int_h^\infty (t^{\epsilon - 1} e^{-t} - (t-h)^{\epsilon - 1} e^{-(t-h)}) e^{t\Delta} dt + \Gamma(\epsilon)^{-1} \int_0^h t^{\epsilon - 1} e^{-t}e^{t\Delta} dt .$$

Since$||e^{t\Delta}||_{L^1 \rightarrow L^1} = 1$,

$$||(I - e^{h\Delta})(I-\Delta)^{-\epsilon} ||_{L^1 \rightarrow L^1} \le \Gamma(\epsilon)^{-1}\big( \int_0^h  t^{\epsilon - 1} e^{-t}dt + \int_h^\infty ((t-h)^{\epsilon - 1} e^{-(t-h)} -t^{\epsilon - 1} e^{-t})dt\big). $$

Here we have used the positivity of the last integrand. Thus we have
$$||(I - e^{h\Delta})(I-\Delta)^{-\epsilon} ||_{L^1 \rightarrow L^1} \le \Gamma(\epsilon)^{-1} 2\int_0^h t^{\epsilon -1} e^{-t} dt \le 2 h^\epsilon (\epsilon \Gamma(\epsilon))^{-1} = 2h^\epsilon \Gamma(1+\epsilon)^{-1}.$$
\epf
$\blacksquare$ \\
$\mathbf{Proof}$  (continued)

$$|| \mathbb{P}  \int_0^{t-\delta}(I -\Delta)^{r/2} e^{(t- \delta-s)\Delta} u(s) \cdot \nabla u(s) ds - \mathbb{P}  \int_0^t(I -\Delta)^{r/2} e^{(t-s)\Delta} u(s) \cdot \nabla u(s) ds||_1 \le $$

$$  || \mathbb{P}  \int_0^{t-\delta}  (I -\Delta)^{r/2} (e^{(t-s)\Delta} - e^{(t-\delta-s)\Delta} ) u(s) \cdot \nabla u(s) ds||_1 + ||\mathbb{P}\int_{t-\delta} ^t  (I -\Delta)^{r/2}e^{(t-s)\Delta} u(s) \cdot \nabla u(s) ds||_1.  $$

The first term above is 
$$|(|I - e^{\delta\Delta}) (I - \Delta)^{-\epsilon/2} \mathbb{P} \int_0^{t-\delta} (I - \Delta)^{(r+\epsilon)/2} e^{(t-\delta-s)\Delta}u(s)\cdot \nabla u(s) ds||_1.$$

Since $r<1$ we can choose $\epsilon > 0 $ so that $r+ \epsilon <1$.  Then using the lemma and the boundedness of the $L^1$ norm of 

$$\mathbb{P} \int_0^{t-\delta} (I - \Delta)^{(r+\epsilon)/2} e^{(t-\delta-s)\Delta}u(s)\cdot \nabla u(s) ds$$ 
which we have shown above for $r + \epsilon < 1$ (uniformly for $0\le t - \delta \le T$ , $T$ fixed but arbitrary)  we have 
 
$$  || \mathbb{P}  \int_0^{t-\delta}  (I -\Delta)^{r/2} (e^{(t-s)\Delta} - e^{(t-\delta-s)\Delta} ) u(s) \cdot \nabla u(s) ds||_1 \le C\delta^{\epsilon/2}.$$

The remaining term is handled the same way as with $t + \delta $ so that we have 
$$ ||\mathbb{P}\int_{t-\delta} ^t  (I -\Delta)^{r/2}e^{(t-s)\Delta} u(s) \cdot \nabla u(s) ds||_1 \le C \int_{t-\delta}^t s^{- (r+1)/2} ds ||u||_{L^\infty([0,T], L^2))}^2   \le $$
 $$2(1-r)^{-1} \delta^{(1-r)/2} ||u||_{L^\infty([0,T], L^2))}^2  .$$

\end{proof} 
$\blacksquare$

The following corollary is actually a corollary of the \emph{proof} of Theorem \ref{L1onederivative}:

\bc \label{Holder}
$v(\cdot) $ is H{\"o}lder continuous in the norm of $\mathcal{L}^1_r(\mathbb{R}^m)$ for $0\le r < 1$.
\be
||v(t+h) - v(t)||_{\mathcal{L}^1_r}\le c_\alpha |h|^\alpha
\ee
if $0< \alpha < (1-r)/2$, $t, t+h \ge 0, |h| \le 1.$  $c_\alpha$ is bounded uniformly in $t$ for $t$ in any compact interval of $[0,\infty)$. 
\ec
\bpf
See the last two inequalities above which apply in the cases $h = \delta$ and $h =-\delta, \delta >0$.
\epf

\br:
If for $r>0$ we define $h_{r,t}$ by the equation
$$(I -\Delta)^{r/2}K_t (x) = (2\pi)^{-m} t^{-(m+r)/2} \int (t+|\eta|^2)^{r/2} e^{-|\eta|^2} e^{iy\cdot \eta} dy = $$
$$Ct^{-(m+r)/2} h_{r,t}(|y|),$$
with $C$ defined so that
$$ h_{r,t}(w): = \int_0^\infty \int_{-1}^1 e^{isw\lambda}(t+ s^2)^{r/2} s^{m-1} e^{-s^2} (1-\lambda^2)^{(m-3)/2} d\lambda ds, $$
it can be shown that 
$$| h_{r,t}(w)| \le C(1 + |w|)^{-(m +\epsilon)}$$
for $\epsilon \in [0, r)$. Here $C$ is uniformly bounded for $t$ in compacts of $[0,\infty)$.  Thus for $r>0$ the integration by parts in the above proof is not necessary to show $||(I-\Delta)^{r/2}\int_0^t e^{(t-s)\Delta} u(s)\cdot \nabla u(s) ds||_1< \infty$. 
 In fact we have 
 $$||(I-\Delta)^{r/2}K_t||_1 \le c t^{-r/2} \int|h_{r,t}(|y|)|dy \le c't^{-r/2}.$$
 Thus 
 $$||(I-\Delta)^{r/2}\int_0^t e^{(t-s)\Delta} u(s)\cdot \nabla u(s) ds||_1 \le c'\int_0^t (t-s)^{-r/2}||u(s)||_2||\nabla u(s)||_2ds \le $$
 $$ c' ||u||_{L^\infty([0,T]; L^2)} (T^{1-r}/(1-r))^{1/2}( \int_0^T ||\nabla u(s)||^2 ds)^{1/2}$$
 where $T \ge t$.
\er

The following proposition is a simple consequence of Theorem \ref{L1onederivative}, Corollary \ref{Holder}, and interpolation:

\bp
$v(t) \in \mathcal{L}^p_s(\mathbb{R}^m)$ and  $v(\cdot)$ is  H{\"o}lder continuous in the $\mathcal{L}^p_s(\mathbb{R}^m)$ norm whenever $1\le p < m/(m-1 +s)$.  Explicitly
given $s\ge 0$  and $p$ so that  $1\le p < m/(m-1 +s)$ then
$$||v(t+h) - v(t)||_{\mathcal{L}^p_s} \le C_\epsilon  |h|^\epsilon$$
for $0 < \epsilon < (m/2)(p^{-1} - p_0^{-1}), p_0 = m/(m- 1 + s)$.  If $1\le p <2$, $v(t) \in L^p$ and $v(\cdot)$ is H{\"o}lder continuous in the $L^p$ norm:
$$||v(t+h) - v(t)||_p  \le C'_\epsilon |h|^\epsilon$$
for $0 < \epsilon< p^{-1} -2^{-1} $.  Note that if $s=0$,  $p^{-1} -2^{-1}  - (m/2)(p^{-1} - p_0^{-1}) = \big ((m/2) - 1\big)p'^{-1} \ge 0.$ 
\ep

\bpf
From the representation $(1-\Delta)^{-w/2}f(x)= \int k_w(x-y) f(y)dy$ with 
$$k_w(x) = \Gamma(w/2)^{-1} \int_0^\infty e^{-t} t^{(w-2)/2} K_t(x) dt$$ for $w>0$,
it is easy to derive the bound $k_w(x)\le C_we^{-|x|/2} |x|^{-(m-w)}$ for $0<w <1$.  
Given $p$ and $s \ge 0$ with $1\le p < p_0 $,  we choose $0 < 1- r  < m(p^{-1} - p_0^{-1}) $ or what is the same $p < m/(m-r+ s))$.  We have
$$ ||(I- \Delta)^{s/2}( v(t+h) - v(t))||_p = ||(I - \Delta)^{-(r-s)/2}(I-\Delta)^{r/2} (v(t+h) - v(t))||_p.$$
Using Minkowski's inequality we have 
$$ ||(I- \Delta)^{s/2}( v(t+h) - v(t))||_p \le ||k_{r-s}||_p ||(I-\Delta)^{r/2} (v(t+h) - v(t))||_1. $$
Clearly $||k_{r-s}||_p < \infty $ if $p(m-(r-s)) < m$ which combined with Corollary \ref{Holder} gives the first result.  The fact that $v(t) \in \mathcal{L}^p_s$ follows from a similar interpolation. \\
To prove the second inequality, for  $1\le p < 2$ we interpolate to obtain with  $\theta = 2(p^{-1} - 2^{-1})$
$$||v(t+h) - v(t)||_p \le ||v(t+h) - v(t)||_1^\theta ||v(t+ h) - v(t)||_2^{1-\theta}\le ||v(t+h) - v(t)||_1^\theta (3||u||_{L^\infty ([0,T];L^2)})^{1-\theta}$$ 
for $T$ large enough. The desired H{\"o}lder continuity follows from Corollary \ref{Holder}.  Basically the same interpolation gives $v(t) \in L^p$.
\epf
$\blacksquare$

The somewhat abstract condition of belonging to the space $\mathcal{L}_p^r(\mathbb{R}^m)$ has a more down to earth consequence.  We prove the following H\"{o}lder condition:  For $|h| \le 1$
\begin{proposition}\label{spaceHolder}
\begin{equation}
||v(t)(\cdot + h) - v(t)(\cdot)||_p  \le c(r) |h|^r ||v(t)||_{ \mathcal{L}_p^r}.
\end{equation}
where $1\le p < m/(m-1 +r)$ and $0 \le r < 1$.  
\end{proposition}
\bpf
This is a result which has nothing to do with the Navier-Stokes equation and is probably well-known.  We give a simple proof a result which holds for $1\le p < \infty$ and $0<r$.
Define $(U(h)w)(x) = w(x+h)$.  We will estimate the norm of $(U(h) - I)(-\Delta + I)^{-r/2}$ as a self - map of $L^p(\mathbb{R}^m)$. We consider the integral kernel of this operator
$$(U(h) -I)(-\Delta + I)^{-r/2}(x,y)  = \Gamma(r/2)^{-1} \int_0^\infty e^{-t} t^{r/2}(K_t(x+h-y) - K_t(x-y)) dt/t.$$  We will use Minkowski's inequality in integral form so we calculate with $y = \frac{x}{2\sqrt t} $ and $h = (\lambda,0,....0)$ with $\lambda \ge 0$ 

$$\int |K_t(x+h) - K_t(x)|dx = (4\pi t)^{-m/2} \int | e^{-|x+h|^2/4t} - e^{-|x|^2/4t} |dx = \pi^{-m/2}\int e^{-|y +\frac {h}{2\sqrt t}|^2} - e^{-|y|^2}| dy = $$ $$\pi^{-1/2} \int_{-\infty}^\infty |e^{-(s+w/2)^2} - e^{-(s-w/2)^2}|ds.$$
where $w = \frac{\lambda}{2\sqrt t}$.  Thus

$$\int |K_t(x+h) - K_t(x)|dx  = 2\pi^{-1/2}\int_0^\infty \int _0^1 (d/dt) [e^{-(s-tw/2)^2} - e^{-(s+tw/2)^2}] dtds = $$ $$w\pi^{-1/2}\int _0^1 \int_0^\infty [(s-tw/2)e^{-(s-tw/2)^2} + (s+tw/2) e^{-(s+tw/2)^2}]ds dt.$$
We have $\int_0^\infty (s-tw/2)  e^{-(s-tw/2)^2}ds = \int_{-tw/2} ^\infty ue^{-u^2}du \le \int_0^\infty u e^{-u^2} du =1/2$, while similarly $\int_0^\infty (s+tw/2) e^{-(s+tw/2)^2} ds \le 1/2$.   Since we also have $\int |K_t(x+h) - K_t(x)|dx \le 2$ we obtain

$$\int |K_t(x+h) - K_t(x)|dx  \le c w(1+w)^{-1}$$.

Finally 

$$||(U(h) -I)(-\Delta + I)^{-r/2}||_{L^p \rightarrow L^p} \le c\int_0^\infty e^{-t} t^{r/2} \frac{\frac{\lambda}{2 \sqrt t}}{ 1 + \frac{\lambda}{2 \sqrt t}} \frac {dt}{t}$$

This integral is easy to estimate.  We find

$$||(U(h) -I)(-\Delta + I)^{-r/2}||_{L^p \rightarrow L^p} \le c(r) |h|^r; r < 1$$ 
$$\le c |h| \log |h|^{-1} ; r = 1$$
$$\le c(r) |h| ; r> 1$$

\epf

\section{Acknowledgement}
Thanks go to my colleague Zoran Grujic for pointing out that a version of Theorem \ref{equivalence} is proved in \cite {LR} and in \cite{FLRT}.

 \section{Appendix}
 
 In this appendix we derive bounds on $(I -\mathbb{P})_{i,j} (I - \Delta)^{r/2} K_t(x)$ where $K_t$ is the integral kernel for the operator $e^{t\Delta}$ in $\mathbb{R}^m$.  Explicitly, $K_t(x-y) = (4\pi t)^{-m/2} e^{-|x-y|^2/4t}$.   We have 
 
 $$(I-\mathbb{P})_{i,j}  (I - \Delta)^{r/2} K_t(x) = ct^{-(m+r)/2} (\omega_i\omega_j g'_{r,t}(|y|) + |y|^{-1}(\delta_{i,j} - \omega_i\omega_j)g_{r,t}(|y|)),$$ 
 
 where $y= x/\sqrt t$, $\omega_i = y_i/|y|$ and

 \be \label{grtdef}
 g_{r,t}(w) =  \int_0^\infty \int_{-1}^1 (t+ s^2)^{r/2} s^{m-2} e^{-s^2} e^{isw\lambda} \lambda (1-\lambda^2)^{(m-3)/2} d\lambda ds.
 \ee
 Here $c$ is an $m$ dependent constant. 
 
 \begin{proposition}
 $g_{r,t} \in C^\infty(\mathbb{R})$ and 
 \begin{equation}
  |g_{r,t}^{(n)}(w)|  \le C_n(1 +|w|)^{-(m+n-1)}
 \end{equation}
 for any $r \ge 0$ uniformly for $t \in [0,T]$, for any $T>0$.
 \end{proposition} 
 
 \begin{proof} 
 
 We easily derive $( d/dw)^ne^{isw\lambda} = w^{-n}\sum_{j=1}^n n_j (sD)^j e^{isw\lambda} $ where the $n_j$ are integers and $D = d/ds$. Integrating by parts in the $s$ integral 
 we obtain 
 $$ g_{r,t}^{(n)}(w)  = w^{-n} \int_{-1}^1 \int_0^\infty e^{isw\lambda} h(s) \lambda (1-\lambda^2)^{(m-3)/2} ds d\lambda$$
 where
 $$h(s) = \sum_{j=1}^nn_j (-Ds)^j [(t+s^2)^{r/2} s^{m-2} e^{-s^2}] $$
 and $D$ is $d/ds$.  In the following we use $|D^l (t+ s^2)^{r/2}| \le C_l (t+s^2)^{(r-l)/2}$ with $C_l = C_l(t)$ bounded for $t$ in compacts of $[0,\infty)$.
 
 We integrate by parts $2l$ times using $[(-w^2 \lambda s)^{-1} (\partial^2/\partial s\partial \lambda)]^l e^{isw\lambda} =  e^{isw\lambda} $ and $ [- (d/d \lambda)\lambda^{-1}]  \lambda(1-\lambda^2)^{k/2} = k \lambda (1-\lambda^2)^{(k-2)/2}$.  We obtain
 \be \label{ltimes}
 g_{r,t}^{(n)}(w) = w^{-2l-n} c_{l,m}\int_0^\infty\int_{-1}^1 e^{isw\lambda} (Ds^{-1})^l h(s)  \lambda (1-\lambda^2)^{- l + (m-3)/2 } d\lambda ds. 
 \ee

 If $m$ is odd we take $l = (m-3)/2$ and obtain
 
 $$g_{r,t}^{(n)} (w) = w^{-(m+n-3)}c_m\int_0^\infty\int_{-1}^1 e^{isw\lambda} (Ds^{-1})^{(m-3)/2} h(s) \lambda d\lambda ds.$$
 
 Integrating by parts once more we have
  $$g_{r,t}^{(n)}(w) = 2iw^{-(m+n -1)}c_m \int_0^\infty \frac{\sin(ws)}{s}  D(Ds^{-1})^{(m-3)/2} h(s) ds.$$
  
  To see that the above integral is bounded uniformly in $w$ we write $\sin(ws)/s = (d/ds) \int_0^{ws} \frac {\sin \tau}{\tau} d\tau$ and integrate by parts one final time to obtain
  
  \begin{equation}\label{grt}
  g_{r,t}^{(n)}(w) =- 2iw^{-(m+n-1)} c_m\int_0^\infty (\int_0^{sw} \frac{\sin\tau}{\tau} d\tau)  D^2(Ds^{-1})^{(m-3)/2} h(s)ds
 \end{equation}
 If $r=0$ then $ D^2(Ds^{-1})^{(m-3)/2} h(s)  = $ (polynomial) $e^{-s^2}$.
 
 If $r> 0$ and $t=0$ terms of the form $ s^{-1 + r}  e^{-s^2}$ arise which look singular for small $r$ at $s = 0$.  This is because we are not keeping track of coefficients.  We do not lose any information if we keep $r$ away from $0$ since bounds on $||(I - \Delta)^{r/2} f||_1 $ for small $r\ge 0 $ follow from those for larger $r$.  It is not hard to see from (\ref{grt}) that $g_{r,t}^{(n)}(w)  = O(w^{-(m +n -1)}) $ uniformly for $t$ in compact subsets of $[0,\infty)$.  To see the situation more explicitly it might help to note that $(Ds^{-1})^l = \sum_{j=1}^l m_j D^j s^{-(2l - j)}$ for certain integers $m_j$.

 Now we turn to even $m$.  Using (\ref{ltimes}) with $l=(m-2)/2$ we find
 $$ g_{r,t}^{(n)} (w) = w^{-(m +n-2) }c_m' \int_0^\infty\int_{-1}^1 e^{isw\lambda} (Ds^{-1})^{(m-2)/2} h(s)  \lambda (1-\lambda^2)^{- 1/2} d\lambda ds .$$
 
Another integration by parts in the $s$ integral gives 

$$ g_{r,t}^{(n)}(w) = w^{-(m+n-1)}ic_m' \int_0^\infty\int_{-1}^1 e^{isw\lambda} D(Ds^{-1})^{(m-2)/2}h(s)(1-\lambda^2)^{- 1/2} d\lambda ds + $$
 $$ + w^{-(m +n-1)} i c_{m}' [(Ds^{-1})^{(m-2)/2} h(s)]|_{s=0} \int_{-1}^1(1-\lambda^2)^{-1/2} d\lambda$$
 
 so that $g_{r,t}^{(n)}(w) = O(w^{-(m+n -1)})$.  From (\ref{grtdef}) it is clear that  $g_{r,t}^{(n)}$ is bounded for all $r\ge 0$ uniformly for $t$ in compacts of $[0,\infty)$.
  \epf

$\blacksquare$

DEPARTMENT OF MATHEMATICS, UNIVERSITY OF VIRGINIA, CHARLOTTESVILLE, VA 22904.\\
\emph{E-mail address:  iwh@virginia.edu}

\end{document}